\documentclass[11pt]{amsart}

\usepackage{amsmath,amsfonts}
\usepackage{amssymb}
\usepackage{amscd}
\usepackage{amsthm}
\usepackage{yhmath}
\usepackage{subfigure}

\usepackage[all]{xy}
\usepackage{color}

\usepackage{marginnote}

\numberwithin{equation}{section}

\newtheorem{theorem}[equation]{Theorem}
\newtheorem{lemma}[equation]{Lemma}
\newtheorem{proposition}[equation]{Proposition}
\newtheorem{corollary}[equation]{Corollary}

\theoremstyle{definition}
\newtheorem{definition}[equation]{Definition}

\theoremstyle{remark}
\newtheorem{remark}[equation]{Remark}

\def\XXint#1#2#3{{\setbox0=\hbox{$#1{#2#3}{\int}$}
	\vcenter{\hbox{$#2#3$}}\kern-.5\wd0}}

\newcommand{\N}{\mathbb N}

\newcommand{\R}{\mathbb R}

\newcommand{\card}{\operatorname{Card}}

\newcommand{\Span}{\operatorname{span}}

\newcommand\carset{1\negmedspace{\rm l}}

\begin{document}

\title[Besicovitch covering 
property in Carnot groups of step 3 and higher]{Remarks about Besicovitch covering 
property in Carnot groups of step 3 and higher}

\author{Enrico Le Donne}

\address[Le Donne]{Department of Mathematics and Statistics, P.O. Box 35,
FI-40014,
University of Jyv\"askyl\"a, Finland}
\email{ledonne@msri.org}

\author{S\'everine Rigot}
\address[Rigot]{Laboratoire de Math\'ematiques J.A. Dieudonn\'e UMR CNRS 7351,  Universit\'e Nice Sophia Antipolis, 06108 Nice Cedex 02, France}
\email{rigot@unice.fr}

\thanks{The work of S.R. is supported by the ANR-12-BS01-0014-01 Geometrya.}

\renewcommand{\subjclassname}{
28C15, 
49Q15, 
43A80. 
}

\date{\today}


\begin{abstract} 
We prove that the Besicovitch Covering Property (BCP) does not hold for some classes of homogeneous quasi-distances on Carnot groups of step 3 and higher. As a special case we get that, in Carnot groups of step 3 and higher, BCP is not satisfied for those homogeneous distances whose unit ball centered at the origin coincides with a Euclidean ball centered at the origin. This result comes in constrast with the case of the Heisenberg groups where such distances satisfy BCP.
\end{abstract}

\maketitle

 

\section{Introduction} \label{section:intro}

Covering theorems, among which is the Besicovitch Covering Property (BCP), see Definition~\ref{def:bcp} below, are known to be some of the fundamental tools of measure theory. More generally they turn out to be classical tools that play a crucial role in many problems in analysis and geometry. We refer for example to \cite{heinonen} and \cite{mattila} for a more detailed introduction about covering theorems.

In contrast to the Euclidean case, the Heisenberg groups equipped with the commonly used (Cygan-)Kor\'anyi  and Carnot-Carath\'eodory distances  are known not to satisfy BCP (\cite{KoranyiReimann}, \cite{Rigot}, \cite{SawyerWheeden}). However, it has been recently proved that BCP holds in the Heisenberg groups equipped with those homogeneous distances whose unit ball centered at the origin coincides with a Euclidean ball centered at the origin (\cite{ledonne-rigot}, see also Theorem~\ref{thm:bcp-heisenberg} below).

For more general Carnot groups, BCP does not hold for Carnot-Carath\'eodory distances (\cite{Rigot}). Motivated by the question of whether one can find homogeneous (quasi-)distances on a given Carnot group for which BCP holds, we prove in the present paper that BCP does not hold for some classes of homogeneous quasi-distances on Carnot groups of step 3 and higher, see Theorem~\ref{thm:main}. As a particular case, we get that the main result in \cite{ledonne-rigot} do not extend to Carnot groups of step 3 and higher, that is, BCP is not satisfied when these groups are equipped with a homogeneous distance whose unit ball centered at the origin coincides with a Euclidean ball centered at the origin, see Corollary~\ref{cor:main}.

To state our results, we first recall the Besicovitch Covering Property in the general quasi-metric setting. Given a nonempty set $X$, we say that $d:X\times X \rightarrow [0,+\infty[$ is a quasi-distance on $X$ if it is symmetric, $d(p,q) =0$ if and only if $p=q$, and there exists a constant $C\geq 1$ such that $d(p,q) \leq C(d(p,p') + d(p',q))$ for all $p$, $p'$, $q\in X$ (quasi-triangle inequality with multiplicative constant $C$). We call $(X,d)$ a quasi-metric space. When speaking of a ball $B$ in $(X,d)$, it will be understood that $B$ is a set of the form $B=B_{d}(p,r)$ for some $p\in X$ and some $r>0$ where $B_{d}(p,r) := \{q\in X;\; d(q,p)\leq r\}$. Note that when $d$ satisfies the quasi-triangle inequality with a multiplicative constant $C=1$, then $d$ is nothing but a distance on $X$.

\begin{definition}[Besicovitch Covering Property] \label{def:bcp}
Let $(X,d)$ be a  quasi-metric space. We say that $(X,d)$ satisfies the Besicovitch Covering Property (BCP) if there exists a constant $N\in\N$ such that the following holds. Let $A$ be a bounded subset of $X$ and $\mathcal{B}$ be a family of balls such that  each point of $A$ is the center of some ball of $\mathcal{B}$, then there is a finite or countable subfamily $\mathcal{F}\subset \mathcal{B}$ such that the balls in $\mathcal{F}$ cover $A$ and every point in $X$ belongs to at most $N$ balls in $\mathcal{F}$, that is, 
\begin{equation*}
\carset_A \leq \sum_{B \in \mathcal{F}} \carset_B \leq N
\end{equation*}
where $\carset_A$ denotes the characteristic function of the set $A$.
\end{definition}

The Besicovitch Covering Property originates from the work of Besicovitch (\cite{B1}, \cite{B2}). It is satisfied in the Euclidean space and more generally in any finite dimensional normed vector space. 

Next, we recall the definition of Carnot groups and state the conventions and notations we shall use throughout this paper. A Carnot group $G$ of step $s\geq 1$ is a connected and simply connected Lie group whose Lie algebra $\mathfrak{g}$ is endowed with a stratification, $\mathfrak{g}  = V_1 \oplus \cdots  \oplus V_s$ where $[V_1,V_j] = V_{j+1}$ for $1\leq j \leq s-1$ and $[V_1,V_s] = \{0\}$. 

We set $n := \dim \mathfrak{g}$ and consider $(X_1,\cdots,X_n)$ a basis of $\mathfrak{g}$ adapted to the stratification, i.e., $(X_{m_{j-1}+1},\cdots,X_{m_j})$ is a basis of $V_j$ for $1\leq j\leq s$ where $m_0:=0$ and $m_j - m_{j-1}:=\dim V_j$. 

We identify $G$ with $\R^n$ via a choice of exponential coordinates of the first kind. Namely, for Carnot groups, the exponential map $\exp:\mathfrak{g} \rightarrow G$ is a diffeomorphism from $\mathfrak{g}$ to $G$. We then identify $p=\exp(x_1 X_1 + \cdots + x_n X_n)\in G$ with $(x_1,\dots,x_n)$ and, using the Baker-Campbell-Hausdorff formula, the group law is then given by $$(x_1,\dots,x_n) \cdot (x'_1,\dots,x'_n) := (x''_1,\dots,x''_n)$$
where $\exp(x''_1 X_1 + \cdots + x''_n X_n) = \exp(x_1 X_1 + \cdots + x_n X_n) \cdot \exp(x'_1 X_1 + \cdots + x'_n X_n)$.

Dilations $(\delta_\lambda)_{\lambda>0}$ on $G$ are given by $$\delta_\lambda(x_1,\cdots,x_n):=(\lambda^{\alpha_1} x_1, \cdots, \lambda^{\alpha_n} x_n)$$
where $\alpha_i = j$ for $m_{j-1}+1\leq i \leq m_{j}$. These dilations form a one parameter group of group automorphisms.

\begin{definition} \label{def:homogeneous-dist} We say that a quasi-distance $d$ on $G$ is \textit{homogeneous} if $d$ is left-invariant, i.e., $d(p\cdot q,p\cdot q') = d(q,q')$ for all $p$, $q$, $q'\in G$, and one-homogeneous with respect to the dilations, i.e., $d(\delta_\lambda(p),\delta_\lambda(q)) = \lambda \, d(p,q)$ for all $p$, $q\in G$ and all $\lambda>0$.
\end{definition}

Every homogeneous quasi-distance on a Carnot group $G$ induces the topology of the group. Note also that any two homogeneous quasi-distances on a Carnot group $G$ are bi-Lipschitz equivalent. In particular, every homogeneous quasi-distance is bi-Lipschitz equivalent to homogeneous distances.

One can characterize homogeneous quasi-distances by means of their unit ball centered at the origin. Namely, if $d$ is a homogeneous quasi-distance on $G$ and $K:=B_d(0,1)$, then $0$ is in the interior of $K$, $K$ is relatively compact, $K$ is symmetric, i.e., $p\in K$ implies $p^{-1} \in K$, and for all $p\in G$ the set $\{\lambda>0;\; \delta_{1/\lambda}(p) \in K\}$ is a closed sub-interval of $]0,+\infty[$. Conversely, if a subset $K$ of $G$ satisfies these assumptions, then 
\begin{equation} \label{e:def-quasidist}
d(p,q):= \inf\{\lambda >0;\; \delta_{1/\lambda}(p^{-1}\cdot q) \in K\}
\end{equation}
defines a homogeneous quasi-distance on $G$. It is the homogeneous quasi-distance whose unit ball centered at the origin is $K$.~\footnote{It may happen that a homogeneous quasi-distance on a Carnot group is not a continuous function on $G\times G$ with respect to the topological structure of the group. The fact that it induces the topology of the group only means that the unit ball centered at the origin contains the origin in its interior and that it is relatively compact. One can show that the quasi-distance is continuous on $G\times G$ if and only if its unit sphere at the origin is closed.}

In particular, any set $K$ of one of the following forms
\begin{equation} \label{e:def-K-1} 
K:= \{(x_1,\dots,x_n) \in G;\; c_1 |x_1|^{\gamma_1} + \cdots + c_n |x_n|^{\gamma_n} \leq 1\}
\end{equation} 
or
\begin{equation} \label{e:def-K-2} 
K:= \{x \in G;\; c_1 \|\overline x_1\|_{d_1}^{\gamma_1} + \cdots + c_s \|\overline x_s\|_{d_s}^{\gamma_s} \leq 1\}
\end{equation} 
for some $\gamma_i >0$, $c_i>0$, induces a homogeneous quasi-distance via \eqref{e:def-quasidist}. In \eqref{e:def-K-2}, for $x=(x_1,\dots,x_n)$, we have set $\overline x_j := (x_{m_{j-1}+1},\dots,x_{m_j})$, $d_j:=\dim V_j$ and $\|\cdot\|_{d_j}$ denotes the Euclidean norm in $\R^{d_j}$.

\medskip

Our main result is the following:
\begin{theorem} \label{thm:main}
Let $G$ be a Carnot group of step 3 or higher. Let $K$ be a subset of $G$ given by \eqref{e:def-K-1} or \eqref{e:def-K-2} and let $d$ be the homogeneous quasi-distance induced by $K$ via \eqref{e:def-quasidist}. Then BCP does not hold in $(G,d)$.
\end{theorem}

Examples of homogeneous distances, i.e., satisfying the quasi-triangle inequality with a multiplicative constant $C=1$, which satisfy the assumption of Theorem~\ref{thm:main} have been given by Hebisch and Sikora. They proved in \cite{Hebisch_Sikora} that for any Carnot group $G$, there exists some $\alpha^*>0$ such that, for all $0<\alpha<\alpha^*$, the Euclidean ball $\{(x_1,\cdots,x_n) \in G;\; |x_1|^{2} + \cdots + |x_n|^{2} \leq \alpha^2\}$ with radius $\alpha$ induces a homogeneous distance on $G$ via \eqref{e:def-quasidist}. For these distances, we have the following corollary. 

\begin{corollary} \label{cor:main}
Let $G$ be a Carnot group of step 3 or higher and let $d$ be a homogeneous distance on $G$ whose unit ball centered at the origin is a Euclidean ball centered at the origin. Then BCP does not hold in $(G,d)$.
\end{corollary}

As already mentioned, such homogeneous distances were our initial motivation and this corollary comes in constrast with the case of the Heisenberg groups, that are Carnot groups of step 2, due to the following result. 

\begin{theorem}[\cite{ledonne-rigot}] \label{thm:bcp-heisenberg}
Let $\mathbb{H}^n$ be the n-th Heisenberg group and let $d$ be a homogeneous distance on $\mathbb{H}^n$ whose unit ball centered at the origin is a Euclidean ball centered at the origin. Then BCP holds in $(G,d)$.
\end{theorem}

To our knowledge, the case of the Heisenberg groups are the only known examples of non abelian Carnot groups for which one can find some homogeneous distances satisfying BCP, and the only known such distances are those considered in Theorem~\ref{thm:bcp-heisenberg}. This makes the Heisenberg groups very peculiar cases as far as the validity of BCP for some homogeneous distance on a Carnot group is concerned. Theorem~\ref{thm:main} indeed shows that natural analogues of these distances do not satisfy BCP as soon as the step of the group is 3 or higher. There are moreover some hints towards the fact that BCP does not hold for any homogeneous distance as soon as the step of the group is 3 or higher, and even in more general graded groups. This will be studied in a forthcoming paper.

The proof of Theorem~\ref{thm:main} (see Section~\ref{section:proofmain}) is based on two main ingredients. First, we show that for a Carnot group equipped with a homogeneous quasi-distance whose unit ball centered at the origin is given by \eqref{e:def-K-1} (respectively \eqref{e:def-K-2}), the validity of BCP implies that $\gamma_1,\dots,\gamma_{m_1}$ (respectively $\gamma_1$) are bounded below by the step of the group, see Lemma~\ref{lem:power}. Hence, for Carnot groups of step 3 and higher, we get $\gamma_1,\dots,\gamma_{m_1} \geq 3$ (respectively $\gamma_1 \geq 3$) whenever BCP holds. Next, a reduction argument on the step of the group by taking a quotient allows us to reduce the problem to the case of the first Heisenberg group equipped with a homogeneous quasi-distance inherited from the original one. The fact that the quotient map is a submetry plays a key role here. Submetries are indeed particularly well adapted tools in this context. See Section~\ref{sect:submetries} where we prove some of their properties related to the Besicovitch Covering Property. On the other hand, we know by \cite{ledonne-rigot} that, in the Heisenberg groups, BCP does not hold for homogeneous quasi-distances whose unit sphere has vanishing Euclidean curvature at the poles. In particular, BCP cannot hold for the inherited homogeneous quasi-distance when $\gamma_1,\dots,\gamma_{m_1} \geq 3$ (respectively $\gamma_1 \geq 3$). This implies in turn that BCP was not satisfied by the original distance.

\section{Weak Besicovitch Covering Property and Submetries} \label{sect:submetries}

First, we introduce what we call here the Weak Besicovitch Covering Property (the terminology might not be standard). 

\begin{definition}[Family of Besicovitch balls]\label{def:BesicovitchBalls}
Let $(X,d)$ be a  quasi-metric space. We say that a family $\mathcal{B}$ of balls in $(X,d)$ is a {\em  family of Besicovitch balls} if $\mathcal{B} = \{B=B_{d}(x_B,r_B)\}$ is a finite family of balls such that $x_B \not \in B'$ for all $B$, $B'\in \mathcal{B}$, $B\not=B'$, and for which $\bigcap_{B\in \mathcal{B}} B \not= \emptyset$.
\end{definition} 

\begin{definition}[Weak Besicovitch Covering Property] \label{def:wbcp}
We say that a quasi-metric space $(X,d)$ satisfies the Weak Besicovitch Covering Property (WBCP) if there exists a constant $Q\in \N$ such that $\card \mathcal{B} \leq Q$ for every family $\mathcal{B}$ of Besicovitch balls in $(X,d)$.
\end{definition}

If $(X,d)$ satisfies BCP, then $(X,d)$ satisfies WBCP. One can indeed take $Q=N$ where $N$ is given by Definition~\ref{def:bcp}.~\footnote{WBCP is in general strictly weaker than BCP. However, BCP and WBCP are equivalent in doubling metric spaces.} We will prove in Section~\ref{section:proofmain} that WBCP, and hence BCP, does not hold in Carnot groups of step 3 and higher equipped with homogeneous quasi-distances as in Theorem~\ref{thm:main}.

Submetries will play a key role in our arguments. In the rest of this  section, we recall the definition of submetries and prove some of their properties to be used in the proof of Theorem~\ref{thm:main}, see Proposition~\ref{prop:bcp-submetry} and Corollary~\ref{cor:quotient-submetry}.

\begin{definition}[Submetry]
Let $(X,d_X)$ and $(Y,d_Y)$ be quasi-metric spaces. We say that $\pi:X\rightarrow Y$ is a submetry if $\pi$ is a surjective map such that 
\begin{equation} \label{e:def}
\pi(B_{d_X}(p,r)) = B_{d_Y}(\pi(p),r)
\end{equation}  
for all $p\in X$ and all $r>0$.
\end{definition}

\begin{remark} \label{rmk:1lip}
Any submetry $\pi:(X,d_X)\rightarrow (Y,d_Y)$ is 1-Lipschitz. Indeed, given $p, q\in X$, set $r:=d_X(p,q)$. We have $q\in B_{d_X}(p,r)$ hence $\pi(q) \in\pi(B_{d_X}(p,r)) = B_{d_Y}(\pi(p),r)$. Hence $d_Y(\pi(p),\pi(q)) \leq r =  d_X(p,q)$.
\end{remark}

The following characterization of submetries will be technically convenient. For subsets $A, B\subset X$, we consider here the distance $d_X(A,B)$ defined by $d_X(A,B):=\inf (d_X(p,q);\; p\in A, \; q\in B)$.

\begin{proposition} \label{charact-submetry}
Let $(X,d_X)$ and $(Y,d_Y)$ be quasi-metric spaces. Let $\pi:(X,d_X)\rightarrow (Y,d_Y)$ be a surjective map. Then the following are equivalent:

 (i) $\pi$ is a submetry,

 (ii) for all $\hat{p} \in Y$, all $\hat{q} \in Y$ and all $p\in \pi^{-1}(\hat p)$, there exists $q\in \pi^{-1}(\hat q)$ such that 
$$d_X(p,q) = d_Y(\hat{p},\hat{q}) = d_X(\pi^{-1}(\hat p),\pi^{-1}(\hat q))  = d_X(p,\pi^{-1}(\hat q))~.$$
\end{proposition}

\begin{proof}
Assume that $\pi$ is a submetry. Let $\hat{p} \in Y$, $\hat{q} \in Y$ and $p\in \pi^{-1}(\hat p)$. Since $\pi$ is 1-Lipschitz (see Remark~\ref{rmk:1lip}), we have $d_Y(\hat p , \hat q) \leq d_X(p',q')$ for  all $p'\in \pi^{-1}(\hat p)$ and all $q'\in \pi^{-1}(\hat q)$. It follows that 
\begin{equation*}
d_Y(\hat p , \hat q) \leq d_X(\pi^{-1}(\hat p), \pi^{-1}(\hat q)) \leq d_X(p,\pi^{-1}(\hat q)).
\end{equation*}
Set $r:=d_Y(\hat{p},\hat{q})$. We have $\hat{q}\in B_{d_Y}(\hat{p},r) = \pi(B_{d_X}(p,r))$ hence one can find  $q\in \pi^{-1}(\hat q) \cap B_{d_X}(p,r)$. Then we have $d_X(p,q) \leq r = d_Y(\hat{p},\hat{q})$. All together, we get that
\begin{equation*} 
d_Y(\hat{p},\hat{q}) \leq d_X(\pi^{-1}(\hat{p}),\pi^{-1}(\hat q)) \leq d_X(p, \pi^{-1}(\hat q))  \leq d_X(p,q) \leq d_Y(\hat{p},\hat{q})
\end{equation*}
from which (ii) follows.

Conversely, assume that (ii) holds. Since $\pi$ is assumed to be  surjective, we only need to prove that \eqref{e:def} holds. Let $p\in X$ and $r>0$. Let us first prove that $\pi(B_{d_X}(p,r)) \subset B_{d_Y}(\pi(p),r)$. Let $\hat{q} \in \pi(B_{d_X}(p,r))$. Then $\hat{q} = \pi(q)$ for some $q\in B_{d_X}(p,r)$ and it follows from (ii) that 
\begin{equation*}
d_Y(\pi(p),\hat{q})  = d_X(p, \pi^{-1}(\hat q) ) \leq d_X(p,q) \leq r\, ,
\end{equation*}
i.e., $\hat{q} \in B_{d_Y}(\pi(p),r)$. Conversely, let $\hat{q} \in B_{d_Y}(\pi(p),r)$. It follows from (ii) that one can find $q\in \pi^{-1}(\hat q)$ such that 
\begin{equation*}
d_X(p,q) = d_Y(\pi(p),\hat{q}) \leq r\,,
\end{equation*}
i.e., $q\in B_{d_X}(p,r)$. Hence $\hat{q}  = \pi(q) \in \pi(B_{d_X}(p,r))$ and this concludes the proof.
\end{proof}

The next proposition shows that submetries preserve the validity of WBCP.

\begin{proposition} \label{prop:bcp-submetry}
Let $(X,d_X)$ and $(Y,d_Y)$ be quasi-metric spaces. Assume that there exists a submetry from $(X,d_X)$ onto $(Y,d_Y)$. If $(X,d_X)$ satisfies WBCP then $(Y,d_Y)$ satisfies WBCP.
\end{proposition}

\begin{proof}
Let $\pi:(X,d_X)\rightarrow (Y,d_Y)$ be a submetry. Let $\mathcal{\hat B} = \{B=B_{d_Y}(y_B,r_B)\}$ be a family of Besicovitch balls in $(Y,d_Y)$ (see Definition~\ref{def:BesicovitchBalls}). Let $\hat{p} \in \bigcap_{B\in \mathcal{\hat B}} B$ and fix some $p\in \pi^{-1}(\hat{p})$. Using Proposition~\ref{charact-submetry}, for each $B_{d_Y}(y_B,r_B)\in \mathcal{\hat B}$, one can find $x_B\in \pi^{-1}(y_B)$ such that $d_X(p,x_B) = d_Y(\hat{p},y_B)$. It follows that $d_X(p,x_B) \leq r_B$ and hence $p\in \bigcap_{B\in \mathcal{\hat B}} B_{d_X}(x_B,r_B)$. On the other hand, since $\pi$ is 1-Lipschitz (see Remark~\ref{rmk:1lip}), we have $d_X(x_B,x_{B'}) \geq d_Y(y_B,y_{B'}) > \max(r_B,r_{B'})$ for all $B$, $B'\in \mathcal{\hat B}$, $B\not=B'$.
It follows that $\{B_{d_X}(x_B,r_B);\, B\in \mathcal{\hat B}\}$ is a family of Besicovitch balls in $(X,d_X)$. Since $(X,d_X)$ satisfies WBCP, we have $\card \mathcal{\hat B} \leq Q$ for some $Q\in \N$ (see Definition~\ref{def:wbcp}). Hence $(Y,d_Y)$ satisfies WBCP as well.
\end{proof}

In the next proposition we give a sufficient condition on the fibers of a surjective map that allows us to construct on the target space a quasi-distance for which this map is a submetry. 

\begin{proposition} \label{prop:parallelfibers-submetry}
Let $(X,d_X)$ be a quasi-metric space and $Y$ be a nonempty set. Let $\pi:X\rightarrow Y$ be a surjective map. Assume that the fibers of $\pi$ are \textit{parallel},  i.e., assume that for all $\hat{p} \in Y$, all $\hat{q} \in Y$ and all $p\in \pi^{-1}(\hat p)$, one can find $q\in \pi^{-1}(\hat q)$ such that
\begin{equation} \label{e:parallelfibers}
 d_X(\pi^{-1}( \hat{p}),\pi^{-1}(\hat q)) = d_X(p,q)\, .
\end{equation}
Then 
\begin{equation*} 
d_Y(\hat{p},\hat{q}):= d_X(\pi^{-1}(\hat{p}),\pi^{-1}(\hat q))
\end{equation*}
defines a quasi-distance on $Y$ and $\pi$ is a submetry from $(X,d_X)$ onto $(Y,d_Y)$.
\end{proposition}

\begin{proof}
First, let us check that $d_Y$ defines a quasi-distance on $Y$. Assume that $d_Y(\hat{p},\hat{q})=0$. Then, by definition of $d_Y$ and using \eqref{e:parallelfibers}, for all $p\in \pi^{-1}(\hat p)$, one can find $q\in \pi^{-1}(\hat q)$ such that $d_X(p,q) = d_Y(\hat{p},\hat{q})=0$. This implies that $p=q$ and hence $\hat p = \hat{q}$. The fact that $d_Y(\hat{p},\hat{q})=d_Y(\hat{q},\hat{p})$ is obvious from the definition of $d_Y$. Next, we check that $d_Y$ satisfies the quasi-triangle inequality with the same multiplicative constant $C$ as $d_X$. Let $\hat{p}$, $\hat{q}$ and $\hat{p}' \in Y$. Let $p'$ be some point in $\pi^{-1}(\hat{p}')$. Using \eqref{e:parallelfibers}, one can find $p\in \pi^{-1}(\hat p)$ such that $d_Y(\hat{p}',\hat{p}) = d_X(p',p)$. Similarly, one can find $q\in \pi^{-1}(\hat q)$ such that $d_Y(\hat{p}',\hat{q}) = d_X(p',q)$. Then we get that
\begin{equation*}
\begin{split}
d_Y(\hat{p},\hat{q}) = d_X(\pi^{-1}(\hat{p}),\pi^{-1}(\hat q))& \leq d_X(p,q)\\ &\leq C(d_X(p,p') +  d_X(p',q)) = C(d_Y(\hat{p},\hat{p}')+ d_Y(\hat{p}',\hat{q}))\, .
\end{split}
\end{equation*}

Finally, the fact that $\pi$ is a submetry from $(X,d_X)$ onto $(Y,d_Y)$ follows from Proposition~\ref{charact-submetry} together with \eqref{e:parallelfibers}.
\end{proof}

We show in the following proposition that Proposition~\ref{prop:parallelfibers-submetry} can be applied to quotient maps from a topological group modulo a boundedly compact normal subgroup.

\begin{proposition} \label{prop:parrallelfibersforquotient}
Let $G$ be a topological group equipped with a left-invariant quasi-distance $d$ which induces the topology of the group. Let $N\lhd G$ be a normal subgroup of $G$. Assume that $N$ is boundedly compact. Then the cosets, i.e., the fibers of the quotient map $\pi: G \rightarrow G/N$, are parallel.
\end{proposition}

\begin{proof}
Let $\hat p$, $\hat q \in G/N$ and $p\in \pi^{-1}(\hat p)$. Since the quasi-distance on $G$ is left-invariant and $N$ is boundedly compact, any coset is boundedly compact as well. It follows that one can find $q\in \pi^{-1}(\hat q)$ such that $d(p,q) = d(p,\pi^{-1}(\hat q))$. For each $\varepsilon>0$, take $p'\in \pi^{-1}(\hat p)$ and $q'\in\pi^{-1}(\hat q)$ such that $d(p',q')\leq d(\pi^{-1}(\hat p),\pi^{-1}(\hat q)) + \varepsilon$. By left-invariance of $d$ and noting that $p\cdot (p')^{-1}\cdot q' \in \pi^{-1}(\hat q)$, we get that 
\begin{multline*}
\varepsilon + d(\pi^{-1}(\hat p),\pi^{-1}(\hat q)) \geq d(p',q') = d(p,p\cdot (p')^{-1}\cdot q')\\
\geq d(p,\pi^{-1}(\hat q)) = d(p,q) \geq d(\pi^{-1}(\hat p),\pi^{-1}(\hat q))\,.
\end{multline*}
Since $\varepsilon$ is arbitrary, it follows that $d(\pi^{-1}(\hat p),\pi^{-1}(\hat q)) = d(p,q)$ and hence the fibers of the quotient map $\pi$ are parallel.
\end{proof}

The next corollary is a straightforward consequence of Proposition~\ref{prop:parallelfibers-submetry} and Proposition~\ref{prop:parrallelfibersforquotient}.

\begin{corollary} \label{cor:quotient-submetry}
Let $G$ be a topological group equipped with a left-invariant quasi-distance $d$ which induces the topology of the group. Let $N\lhd G$ be a normal subgroup of $G$. Assume that $N$ is boundedly compact. Let $\pi: G \rightarrow G/N$ denote the quotient map. Then 
$$d_{G/N}(\hat{p},\hat{q}):= d(\pi^{-1}(\hat{p}),\pi^{-1}(\hat q))$$
defines a quasi-distance on $G/N$ and $\pi$ is a submetry from $(G,d)$ onto $(G/N,d_{G/N})$.
\end{corollary}

\section{Proof of Theorem \ref{thm:main}} \label{section:proofmain}

This section is devoted to the proof of Theorem~\ref{thm:main}. We consider  a Carnot group $G$ of step $s$ equipped with a homogeneous quasi-distance $d$ whose unit ball centered at the origin is given by \eqref{e:def-K-1}, i.e., can be described as 
$$B_{d}(0,1)=\{(x_1,\cdots,x_n) \in G;\; c_1 |x_1|^{\gamma_1} + \cdots + c_n |x_n|^{\gamma_n} \leq 1\}$$
for some $\gamma_i >0$, $c_i>0$. The case of a homogeneous quasi-distance whose unit ball centered at the origin is given by \eqref{e:def-K-2} is similar and can be proved with the same arguments.

First, we prove that the validity of BCP implies that $\gamma_1,\dots,\gamma_{m_1}$ are bounded below by the step of the group. Recall that $m_1$ denotes the dimension of the first layer $V_1$ of the stratification of the Lie algebra $\mathfrak{g}$ of $G$.

\begin{lemma} \label{lem:power}
Assume that BCP holds in $(G,d)$. Then $\min (\gamma_1,\dots,\gamma_{m_1}) \geq s$.
\end{lemma}

\begin{proof}
Let $1\leq i \leq m_1$ be fixed and set $N_i:=\{(x_1,\cdots,x_n) \in G;\; x_k=0 \text{ for } k\not= i,n\}$. Since BCP, and hence WBCP (see Definition~\ref{def:wbcp}), holds in $(G,d)$, WBCP also holds in $(N_i,d_{N_i})$ where $d_{N_i}$ denotes the quasi-distance $d$ restricted to $N_i$.~\footnote{\label{bcp-subset}More generally, if $Y$ is a subset of a quasi-metric space $(X,d_X)$ which satisfies WBCP, then $(Y,d_Y)$ also satisfies WBCP where $d_Y$ denote the quasi-distance $d_X$ restricted to $Y$.}

On the other hand, $N_i$ is an abelian subgroup of $G$ that can be identified with $\R^2$ equipped with the usual addition, denoted by $+$, as a group law and with the family of dilations $\tilde \delta_\lambda(x,y):=(\lambda x, \lambda^s y)$ for $\lambda>0$. With this identification, the quasi-distance $d_{N_i}$ is then a left-invariant and one-homogeneous quasi-distance on $(\R^2,+,(\tilde\delta_\lambda)_{\lambda>0})$ whose unit ball centered at the origin can be described as $\{(x,y) \in \R^2;\; c_i|x|^{\gamma_i}+c_n|y|^{\gamma_n}\leq 1\}$. It then follows from Lemma~\ref{lemma:R2} below that $\gamma_i\geq s$. 
\end{proof}

\begin{lemma} \label{lemma:R2} Let $\R^2$ be equipped with the usual addition, denoted by $+$, as a group law and the family of dilations $\tilde\delta_\lambda(x,y):=(\lambda x, \lambda^s y)$ for some $s>0$. Let $\rho$ be 
a left-invariant and one-homogeneous quasi-distance on $(\R^2,+,(\tilde\delta_\lambda)_{\lambda>0})$. Assume that the unit ball centered at the origin can be described as
\begin{equation*}
B_\rho(0,1) = \{(x,y)\in\R^2;\, \alpha |x|^a+ \beta |y|^b\leq 1\}
\end{equation*}
for some $a>0$, $b>0$, $\alpha>0$, $\beta>0$. If WBCP holds in $(\R^2,\rho)$, then $a\geq s$.
\end{lemma}

\begin{proof}
First, we note that we only need to consider the case $\alpha=\beta=1$. Indeed, considering the group automorphism $$f(x,y):= (\alpha^{1/a}\; x, \beta^{1/b}\; y)~,$$
which commutes with the dilations $\tilde\delta_\lambda$, then 
$$\rho'(p,q) := \rho(f^{-1}(p),f^{-1}(q))$$ defines a homogeneous quasi-distance on $(\R^2,+,(\tilde\delta_\lambda)_{\lambda>0})$ and $f:(\R^2,\rho) \rightarrow (\R^2,\rho')$, $f^{-1}: (\R^2,\rho') \rightarrow (\R^2,\rho)$ are submetries. It then follows from Proposition~\ref{prop:bcp-submetry} that WBCP holds in $(\R^2,\rho)$ if and only if WBCP holds in $(\R^2,\rho')$.

Thus let us assume that 
\begin{equation*}
B_\rho(0,1) = \{(x,y)\in\R^2;\, |x|^a+|y|^b\leq 1\}~.
\end{equation*}

Arguing by contradiction, let us assume that $0<a<s$. We will prove that one can find $r>1$ and a positive sequence $(\varepsilon_n)_{n\geq 1}$ decreasing to 0 such that, setting $p_n:=(x_n,y_n)$ where
\begin{equation} \label{e:xnyn}
x_n=r^{-n} \quad \text{and} \quad y_n= \varepsilon_n^{-s} \left(1-\varepsilon_n^a \, r^{-na}\right)^{1/b}~,
\end{equation}
the following hold. First,
\begin{equation} \label{e:intersection}
0 \in \partial B_\rho(p_n,\varepsilon_n^{-1})
\end{equation}
for all $n\geq 1$. Second, 
\begin{equation} \label{e:distcenter}
\rho(p_n,p_k) > \varepsilon_n^{-1}
\end{equation}
for all $n\geq 2$ and all $1\leq k \leq n-1$. Since the sequence $(\varepsilon_n^{-1})_{n\geq 1}$ is increasing, we get from \eqref{e:distcenter} that 
\begin{equation*}
\rho(p_n,p_k) > \max(\varepsilon_n^{-1},\varepsilon_k^{-1})
\end{equation*}
for all $n\not= k$ hence $p_k \not \in B_\rho(p_n,\varepsilon_n^{-1})$ for all  $n\not= k$. Combining this with \eqref{e:intersection}, we get that $\{B_\rho(p_n,\varepsilon_n^{-1});\; n\in J\}$ is a family of Besicovitch balls for any $J\subset \N$ finite which gives a contradiction to the validity of WBCP in $(\R^2,\rho)$.

First, it follows from \eqref{e:xnyn} that 
\begin{equation*}
|\varepsilon_n x_n|^a + |\varepsilon_n^s\, y_n|^b = \varepsilon_n^a\, r^{-na} + (1-\varepsilon_n^a \, r^{-na}) = 1
\end{equation*}
hence $\rho(0, \tilde\delta_{\varepsilon_n}(p_n)) = 1$. By homogeneity it follows that $\rho(0, p_n) = \varepsilon_n^{-1}$ hence \eqref{e:intersection} holds for any fixed $r>1$ and any positive sequence $(\varepsilon_n)_{n\geq 1}$. Hence it remains to find some $r>1$ and some positive sequence $(\varepsilon_n)_{n\geq 1}$ decreasing to 0 such that \eqref{e:distcenter} holds to conclude the proof.

Let $r>1$ to be fixed later and set $\varepsilon_1 =1$. By induction, assume that $\varepsilon_1>\cdots > \varepsilon_n$ have already been choosen. We are looking for $\varepsilon_{n+1} \in (0,\varepsilon_n)$ such that $\rho(p_{n+1},p_k) > \varepsilon_{n+1}^{-1}$, i.e.,
\begin{equation} \label{e:distcenterbis}
\rho(\tilde\delta_{\varepsilon_{n+1}}(p_{n+1}),\tilde\delta_{\varepsilon_{n+1}}(p_k))>1
\end{equation}
for all $1\leq k \leq n$. We have 
\begin{multline*}
|\varepsilon_{n+1}(x_{k} - x_{n+1})|^a + |\varepsilon_{n+1}^s (y_{n+1} - y_k)|^b \\ = \varepsilon_{n+1}^a \, (r^{-k} - r^{-(n+1)})^a + |(1-\varepsilon_{n+1}^a \, r^{-(n+1)a})^{1/b} - \varepsilon_{n+1}^s \varepsilon_{k}^{-s} (1-\varepsilon_k^a \, r^{-ka})^{1/b}|^b~.
\end{multline*}
Since $s>a>0$, we have, for all $k\in \{1,\cdots,n\}$ fixed,
\begin{equation*}
 (1-\varepsilon_{n+1}^a \, r^{-(n+1)a})^{1/b} - \varepsilon_{n+1}^s \varepsilon_{k}^{-s} (1-\varepsilon_k^a \, r^{-ka})^{1/b} = 1-b^{-1}\,\varepsilon_{n+1}^a \, r^{-(n+1)a} + o(\varepsilon_{n+1}^a).
\end{equation*}
It follows that
\begin{equation*}
\begin{split}
|\varepsilon_{n+1}(x_{k} - x_{n+1})|^a &+ |\varepsilon_{n+1}^s (y_{n+1} - y_k)|^b  \\
&= 1 + \varepsilon_{n+1}^a \, ((r^{-k} - r^{-(n+1)})^a - r^{-(n+1)a} ) +  o(\varepsilon_{n+1}^a)\\
&= 1 + \varepsilon_{n+1}^a \, r^{-ak} \, ((1 - r^{-(n+1)+k})^a - r^{(-(n+1)+k)a} ) +  o(\varepsilon_{n+1}^a) \\
&\geq 1 + \varepsilon_{n+1}^a \, r^{-ak} \, ((1 - r^{-1})^a - r^{-a} ) +  o(\varepsilon_{n+1}^a)~.
\end{split}
\end{equation*}
Hence, choosing $r>1$ so that $(1 - r^{-1})^a - r^{-a} >0$, we get that one can choose $\varepsilon_{n+1}$ small enough so that 
\begin{equation*}
|\varepsilon_{n+1}(x_{k} - x_{n+1})|^a + |\varepsilon_{n+1}^s (y_{n+1} - y_k)|^b >1
\end{equation*}
for all $1\leq k \leq n$ which proves \eqref{e:distcenterbis}.
\end{proof}

From now on, we assume that $G$ is a Carnot group of step 3 or higher and we argue by contradiction, assuming that BCP holds in $(G,d)$. Hence, we have from Lemma~\ref{lem:power} that 
\begin{equation} \label{e:power}
\min (\gamma_1,\dots,\gamma_{m_1}) \geq 3~.
\end{equation}

Let us consider $N:=\exp (V_3 \oplus \cdots \oplus V_s)$. Then $N$ is a normal subgroup of $G$. The quotient group $\hat{G} := G/N$ can be identified with $\R^{m_2}$ equipped with the group law 
$$\hat p * \hat p' := \hat \pi ([\hat p,0] \cdot [\hat p',0])$$
where $\hat \pi (x_1,\dots,x_n) := (x_1,\dots,x_{m_2})$ and where, for   $p=(x_1,\dots,x_{n})\in G$, we set $p:=[\hat p,\tilde p]$ with $\hat p := \hat \pi (p)$ and $\tilde p := (x_{m_2+1},\dots,x_n)$.

The group $\hat G$ inherits from $G$ a natural structure of Carnot group of step 2 with dilations given by $$\hat \delta_\lambda (x_1,\dots,x_{m_{2}}) :=(\lambda \, x_1, \cdots, \lambda \, x_{m_1}, \lambda^2\,  x_{m_1+1},\cdots, \lambda^2 \, x_{m_2})~.$$

Since the exponential map is here a global diffeomorphism, $N$ is boundedly compact in $(G,d)$ and it follows from Corollary~\ref{cor:quotient-submetry} that $$d_{\hat G}(\hat p, \hat q) := d(\hat{\pi}^{-1}(\hat p), \hat{\pi}^{-1}(\hat q))$$
defines a quasi-distance on $\hat G$ and $\hat \pi : (G,d) \rightarrow (\hat G,d_{\hat G})$ is a submetry. Hence, by Proposition~\ref{prop:bcp-submetry}, $(\hat G,d_{\hat G})$ satisfies WBCP.

Let us now check that $d_{\hat G}$ is the homogeneous quasi-distance on $\hat G$ whose unit ball centered at the origin is given by
\begin{equation} \label{e:distquotient}
B_{d_{\hat G}}(0,1)=\{(x_1,\dots,x_{m_{2}}) \in \hat G;\; c_1 |x_1|^{\gamma_1} + \cdots + c_{m_2} |x_{m_{2}}|^{\gamma_{m_2}} \leq 1\}~.
\end{equation}

The fact that $d_{\hat G}$ is left-invariant can be easily checked using the left invariance of $d$ and the fact that $\hat \pi$ is a group homomorphism. The homogeneity of $d_{\hat G}$ with respect to dilations $\hat\delta_\lambda$ can be easily checked as well noting that $\hat \pi^{-1}(\hat \delta_\lambda(\hat{p}))=  \delta_\lambda (\hat \pi^{-1}(\hat{p}))$ for all $\hat{p}\in \hat{G}$ and $\lambda>0$ and using the homogeneity of $d$.

Let us now check that \eqref{e:distquotient} holds. For $p=(x_1,\dots,x_n) \in G$, we have
\begin{equation*}
\begin{split}
d(0,p) &= \inf \{r>0;\, \delta_{1/r}(p) \in B_{d}(0,1)\}\\
& = \inf \{r>0;\, f(r^{-\alpha_1}|x_1|,\dots,r^{-\alpha_n}|x_n|)\leq 1\}
\end{split}
\end{equation*}
where $f:(\R^+)^n \rightarrow \R^+$ is given by $f(t_1,\dots,t_n) := c_1 \, t_1^{\gamma_1} + \cdots + c_n \, t_n^{\gamma_n}$. This function being  increasing with respect to the $(n-m_{2})$ last coordinates, we have
\begin{equation*}
d(0,[\hat \pi (p),0]) \leq d(0,p)
\end{equation*}
for all $p\in G$. Together with Proposition~\ref{charact-submetry}, this implies that 
\begin{equation*}
d_{\hat G}(0,\hat{p}) = d(0,\hat\pi^{-1}(\hat{p})) = d(0,[\hat p,0])
\end{equation*}
for all $\hat{p} \in \hat G$. Hence,
 \begin{equation*}
 B_{d_{\hat G}}(0,1) =\{(x_1,\dots,x_{m_2})\in \hat G;\; f(|x_1|,\dots,|x_{m_2}|,0,\dots,0)\leq 1\}
 \end{equation*}
which proves \eqref{e:distquotient}.

Let $(Y_1,\dots,Y_{m_2})$ be the basis of the Lie algebra $\hat{\mathfrak{g}}$ of $\hat G$ inherited from the choosen basis $(X_1,\cdots,X_n)$ adapted to the stratification of $\mathfrak{g}$. Let us fix $i,j\in\{1,\dots,m_1\}$ such that $[Y_i,Y_j]\not=0$. Set $\mathfrak{h}:= \Span(Y_i,Y_j,[Y_i,Y_j])$ and $H:= \exp \mathfrak{h}$. Then $H$ is a subgroup of $\hat G$ that can be identified with the first Heisenberg group. Recall that the first Heisenberg group is the Carnot group of step 2 whose stratification of the Lie algebra is given by $ W_1 \oplus W_2$ where $\dim W_1=2$ and $\dim W_2=1$. Hence we can identify $H$ with $\R^3$ equipped with the Heisenberg group structure given by 
$$(x,y,z)\cdot (x',y',z') := (x+x',y+y',z+z'+\frac{1}{2}(x y' - x' y))~,$$ 
where we identify $\exp(x Y_i + y Y_j + z [Y_i,Y_j])$  with $(x,y,z)$, and equipped with the family of dilations $((x,y,z)\mapsto (\lambda x,\lambda y,\lambda^2 z))_{\lambda>0}$.

The quasi-distance $d_H$ induced by the restriction of $d_{\hat G}$ on $H$ is then a homogeneous quasi-distance whose unit ball centered at the origin is given by
$$B_{d_H}(0,1)=\{(x,y,z) \in H;\; c_i |x|^{\gamma_i} + c_j |y|^{\gamma_j} + c_{m_1+1} | \xi_{m_1+1} \; z|^{\gamma_{m_1+1}} + \cdots + c_{m_2} | \xi_{m_2} \; z|^{\gamma_{m_2}} \leq 1\}$$
where $[Y_i,Y_j] = \xi_{m_1+1} Y_{m_1+1} + \cdots +\xi_{m_2} Y_{m_2}$.

Since WBCP holds in $(\hat G,d_{\hat G})$, WBCP also hold in $(H,d_H)$ (see Footnote~\ref{bcp-subset}). On the other hand, we have $\gamma_i, \gamma_j\geq 3$ by \eqref{e:power}. Near the north pole, i.e., the intersection of $\partial B_{d_H}(0,1)$ with the positive $z$-axis, $B_{d_H}(0,1)$ can thus be described as the subgraph $\{(x,y,z)\in H;\; z\leq \varphi (x,y)\}$ of a $C^2$ function $\varphi$ whose first and second order partial derivatives vanish at the origin. Then it follows from \cite[Theorem 6.1]{ledonne-rigot} that WBCP can not hold in $(H,d_H)$. Note that Theorem 6.1 in \cite{ledonne-rigot} holds not only for homogeneous distances but more generally for homogeneous quasi-distances (with the same proof). This gives a contradiction and concludes the proof of Theorem~\ref{thm:main}.

\medskip

\noindent \textbf{Acknowledgement.} The second-named author is very grateful to the Department of Mathematics and Statistics of the University of Jyv\"askyl\"a, where part of this work was done, for its hospitality.

\end{document}